\documentclass[reqno]{amsart}
\usepackage{microtype}
\usepackage{url}
\usepackage{listings}
\usepackage[dvipsnames]{xcolor}
\usepackage{amsfonts, amsthm, amsmath}
\allowdisplaybreaks[4]
\usepackage{rotating}
\usepackage{tikz}
\usepackage{graphics}
\usepackage{amssymb}
\usepackage{amscd}
\usepackage[latin2]{inputenc}
\usepackage{t1enc}
\usepackage[mathscr]{eucal}
\usepackage{indentfirst}
\usepackage{graphicx}
\usepackage{graphics}
\usepackage{pict2e}
\usepackage{mathrsfs}
\usepackage[shortlabels]{enumitem}
\usepackage{CJK}
\usepackage{hyperref}
\hypersetup{colorlinks=true}
\usepackage{cite}
\usepackage{color}
\usepackage{epic}
\usepackage{multirow}
\usepackage{makecell}
\usepackage{algorithmic}
\usepackage[top=3cm,bottom=2cm,right=2.4cm,left=2.4cm]{geometry}

\hypersetup{colorlinks=true, 
	linktoc=all, 
	linkcolor=blue, 
    citecolor=blue} 
\numberwithin{equation}{section}
\newtheorem{theorem}{Theorem}[section]
\newtheorem{lemma}[theorem]{Lemma}

\theoremstyle{definition}

\newtheorem{remark}[theorem]{Remark}

\makeatletter
\renewcommand{\maketag@@@}[1]{\hbox{\normalsize\normalfont#1}}%
\@namedef{subjclassname@2020}{\textup{2020} Mathematics Subject Classification}

\makeatother

\begin{document}
\title[Internal congruences for $k$-colored overpartitions into odd parts]
{Internal congruences modulo powers of 2 for $k$-colored overpartitions into odd parts}

\author[Manjil P. Saikia]{Manjil P. Saikia}
\address[Manjil P. Saikia]{Mathematical and Physical Sciences Division, School of Arts
and Sciences, Ahmedabad University, Navrangpura, Ahmedabad 380009, Gujarat, India}
\email{manjil.saikia@ahduni.edu.in}

\author[Prabal Talukdar]{Prabal Talukdar}
\address[Prabal Talukdar]{Mathematical and Physical Sciences Division, School of Arts
and Sciences, Ahmedabad University, Navrangpura, Ahmedabad 380009, Gujarat, India}
\email{prabaltalukdar89@gmail.com}

\author[Dazhao Tang]{Dazhao Tang}
\address[Dazhao Tang]{School of Mathematical Sciences, Chongqing Normal University,
Chongqing 401331, P.R. China}
\email{dazhaotang@sina.com}

\date{\today}

\begin{abstract}
Let $\overline{po}_k(n)$ denote the number of $k$-colored overpartitions of $n$ into
odd parts. For every odd number of colors, we establish infinite families of internal
congruences modulo high powers of $2$. The proof proceeds by expressing the relevant
difference generating functions as a fixed theta quotient multiplied by polynomials
in an auxiliary function. Uniform lower bounds for the $2$-adic valuations of the
polynomial coefficients then yield the stated internal congruences. These results
generalize an internal congruence of Saikia, Sarma and Sellers (J. Ramanujan Math.
Soc.~2025) on overpartition triples into odd parts.
\end{abstract}



\thanks{}

\maketitle

\section{Introduction}

Throughout this paper, we assume that $q$ is a complex number satisfying $|q|<1$ and
adopt the standard $q$-Pochhammer symbol:
\begin{align*}
(a;q)_\infty:=\prod_{j=0}^\infty(1-aq^j).
\end{align*}

A partition $\lambda$ of a positive integer $n$ is a finite weakly decreasing sequence
of positive integers $\lambda_1\geq\lambda_2\geq\cdots\geq\lambda_r$ such that
$\sum_{i=1}^{r}\lambda_{i}=n$. The numbers $\lambda_i$ are called the parts of the
partition $\lambda$. In order to provide combinatorial proofs of certain classical
$q$-series identities, Corteel and Lovejoy \cite{CL2004} introduced the notion of
overpartitions. An overpartition of $n$ is a partition of $n$ in which the first
(equivalently, the last) occurrence of each distinct part may be overlined. Let
$\overline{p}(n)$ denote the number of overpartitions of $n$, with the convention that
$\overline{p}(0)=1$. The definition gives the generating function
\begin{align*}
\sum_{n=0}^\infty\overline{p}(n)q^n=\dfrac{(-q;q)_\infty}{(q;q)_\infty}.
\end{align*}

We focus on overpartitions into odd parts. Let $\overline{po}(n)$ denote the number of
such overpartitions of $n$. Hirschhorn and Sellers \cite{HS2006} initiated the study
of congruence properties of this function in 2006. Among other results, they established
several congruences modulo $2^a3^b$ for small values of $a$ and $b$. Further results on
its congruence properties can be found in the work of Chen \cite{Chen2014}, Merca
\cite{Mer2022}, Ray and Barman \cite{RB2018}, and Xia \cite{Xia2016}.

A natural extension allows parts to occur in several colors. For a positive integer $k$,
a $k$-colored overpartition into odd parts may be viewed as an ordered $k$-tuple of
overpartitions into odd parts, with empty components allowed. Its size is the sum of the
sizes of its components. Let $\overline{po}_k(n)$ denote the number of $k$-colored
overpartitions of $n$ into odd parts. Since the components are chosen independently,
their generating functions multiply, giving
\begin{align*}
\sum_{n=0}^\infty\overline{po}_k(n)q^n=\dfrac{(-q;q^2)_\infty^k}{(q;q^2)_\infty^k}
=\dfrac{(q^2;q^2)_\infty^{3k}}{(q;q)_\infty^{2k}(q^4;q^4)_\infty^{k}}.
\end{align*}

For two colors, Lin \cite{Lin2012} established several congruences modulo small powers
of $2$ for $\overline{po}_2(n)$. Building on this work, Adiga and Ranganatha
\cite{AR2018} derived infinite families of congruences modulo high powers of $2$.
Further congruence results for $k$-colored overpartitions into odd parts appear in the
work of Ahmed, Barman and Ray \cite{ABR2020}, Das, Saikia and Sarma \cite{DSS2026},
Drema and Saikia \cite{DS2023}, and Saikia, Sarma and Sellers \cite{SSS2025}. More
recently, Tang \cite{Tang2026b} established several infinite families of congruences
modulo powers of $2$ for $\overline{po}_{2k}(n)$. In particular, for any $k\geq1$,
$m\geq0$ and $n\geq1$,
\begin{align}
\overline{po}_{2k}\big(2^{m+3}n\big) &\equiv0\pmod{2^{3m+\epsilon(k)}},
\label{Unif-cong-1}
\end{align}
where
\begin{align*}
\epsilon(k)=\begin{cases}
\;8 &\quad\textrm{if}\;\;k\equiv1\pmod{2},\\[2pt]
11 &\quad\textrm{if}\;\;k\equiv0\pmod{2},
\end{cases}
\end{align*}
and moreover, for any $k\geq1$, $m\geq0$ and $n\geq0$,
\begin{align}
\overline{po}_{2k}\big(2^{m+3}(4n+3)\big) &\equiv0\pmod{2^{3m+2+\epsilon(k)}},
\label{Unif-cong-2}\\
\overline{po}_{2k}\big(2^{m+3}(8n+7)\big) &\equiv0\pmod{2^{3m+3+\epsilon(k)}}.
\label{Unif-cong-3}
\end{align}

The congruences \eqref{Unif-cong-1}--\eqref{Unif-cong-3} concern even numbers of colors.
For odd numbers of colors, we study differences between values at an index and at twice
that index, since these differences can exhibit greater divisibility than the individual
values. For example, $\overline{po}_1(4)=6$, $\overline{po}_1(8)=22$, and hence $\overline{po}_1(8)-\overline{po}_1(4)=16$. This example suggests examining
how the divisibility of these differences changes under repeated extraction of
even-indexed coefficients. This leads us to internal congruences. For a formal power
series $\sum_{n=0}^\infty f(n)q^n$, an internal congruence for $f(n)$ refers to a
congruence relation of the form
\begin{align*}
f(An+B)\equiv\kappa\!\,f(Cn+D)\pmod{M},
\end{align*}
where $\kappa$ is an integer, $M$ is a positive integer, and $An+B$ and $Cn+D$ represent
distinct arithmetic progressions.

Our main theorem compares corresponding terms of two subsequences, with the argument on
one side twice that on the other.

\begin{theorem}\phantomsection\label{Main-THM}
\begin{enumerate}[label=\textup{(\roman*)}, leftmargin=*]
\item For any $k\geq0$, $m\geq1$ and $n\geq0$,
\begin{align}
\overline{po}_{2k+1}\big(2^{m+2}n\big)\equiv\overline{po}_{2k+1}\big(2^{m+1}n\big)
\pmod{2^{m+3}}.\label{Int-cong-fam-1}
\end{align}
\item For any $k\geq0$, $m\geq1$ and $n\geq0$,
\begin{align}
\overline{po}_{2k+1}\big(2^{m+2}(4n+3)\big) &\equiv\overline{po}_{2k+1}\big(2^{m+1}
(4n+3)\big)\pmod{2^{m+4}},\label{Int-cong-fam-2}\\
\overline{po}_{2k+1}\big(2^{m+2}(5n+2)\big) &\equiv\overline{po}_{2k+1}\big(2^{m+1}
(5n+2)\big)\pmod{2^{m+4}},\label{Int-cong-fam-3}\\
\overline{po}_{2k+1}\big(2^{m+2}(5n+3)\big) &\equiv\overline{po}_{2k+1}\big(2^{m+1}
(5n+3)\big)\pmod{2^{m+4}},\label{Int-cong-fam-4}\\
\overline{po}_{2k+1}\big(2^{m+2}(8n+5)\big) &\equiv\overline{po}_{2k+1}\big(2^{m+1}
(8n+5)\big)\pmod{2^{m+5}},\label{Int-cong-fam-5}\\
\overline{po}_{2k+1}\big(2^{m+2}(8n+7)\big) &\equiv\overline{po}_{2k+1}\big(2^{m+1}
(8n+7)\big)\pmod{2^{m+5}}.\label{Int-cong-fam-6}
\end{align}
\end{enumerate}
\end{theorem}

\begin{remark}
Saikia, Sarma and Sellers \cite[Eq.~(38)]{SSS2025} proved that for any $n\geq0$,
\begin{align}
\overline{po}_3(8n)\equiv\overline{po}_3(4n)\pmod{16}.\label{SSS-int-cong}
\end{align}
This is precisely the case $(k,m)=(1,1)$ of \eqref{Int-cong-fam-1}.
\end{remark}

The remainder of this paper is organized as follows. In Section~\ref{Sect:TFL}, we
establish the initial extraction identities for the two residue classes of odd color
counts modulo four, together with a common recurrence for subsequent extractions of
even-indexed coefficients. Section~\ref{Sect:DGF} uses these identities to express the
relevant difference generating functions as a fixed theta quotient multiplied by
polynomials in an auxiliary function. In Section~\ref{Sect:PT}, we derive lower bounds
for the $2$-adic valuations of the polynomial coefficients and use them to prove the
main theorem by extracting the required subsequences. Finally, we ask whether
divisibility by each fixed power of $2$ holds on a set of natural density one for
every fixed odd number of colors.

\section{Three fundamental identities}\label{Sect:TFL}

This section develops the coefficient-extraction identities used throughout the proof.
We begin with the theta functions and auxiliary quantities needed to state them.

For notational convenience, we write
\begin{align*}
E(q)=(q;q)_\infty.
\end{align*}

Ramanujan's general theta function is given by
\begin{align}
f(a,b):=\sum_{n=-\infty}^\infty a^{n(n+1)/2}b^{n(n-1)/2}=(-a;ab)_\infty(-b;ab)_\infty
(ab;ab)_\infty,\label{Def-RGTF}
\end{align}
where $|ab|<1$. The product representation in \eqref{Def-RGTF} is the Jacobi triple product identity \cite[p.~35, Entry~19]{Ber1991}. Two important specializations of
$f(a,b)$ are the classical theta functions
\begin{align}
\varphi(q) &:=f(q,q)=\sum_{n=-\infty}^\infty q^{n^2}=\dfrac{E(q^2)^5}{E(q)^2E(q^4)^2},
\label{Def-phi}\\
\psi(q) &:=f(q,q^3)=\sum_{n=0}^\infty q^{n(n+1)/2}=\dfrac{E(q^2)^2}{E(q)}.
\label{Def-psi}
\end{align}
The infinite product representations in \eqref{Def-phi} and \eqref{Def-psi} follow
immediately from \eqref{Def-RGTF}. Furthermore, replacing $q$ by $-q$ in
\eqref{Def-phi} yields
\begin{align}
\varphi(-q)=\dfrac{E(q)^2}{E(q^2)}.\label{Def-NRGTF}
\end{align}

To simplify subsequent expressions, we introduce the auxiliary functions
\begin{align}
\gamma &=\dfrac{\varphi(-q^2)}{\varphi(-q)}=\dfrac{E(q^2)^3}{E(q)^2E(q^4)},&
\delta &=\dfrac{\varphi(-q^2)^3}{\varphi(-q)^3}=\dfrac{E(q^2)^9}{E(q)^6E(q^4)^3},&
\label{Four-def-1}\\[3pt]
\varrho &=\dfrac{\varphi(q)\varphi(q^2)}{\varphi(-q)^2}=\dfrac{E(q^2)^5E(q^4)^3}
{E(q)^6E(q^8)^2},&
\xi &=q\dfrac{\psi(q^4)^2}{\varphi(-q)^2}=q\dfrac{E(q^2)^2E(q^8)^4}
{E(q)^4E(q^4)^2}.&\label{Four-def-2}
\end{align}
The first two functions provide the basic factors in the initial generating functions
for the two color classes. The remaining two serve as the common theta factor and the
polynomial variable, respectively.

To select the required subsequences, we use the following coefficient-extraction
operator. For a positive integer $m$, define the coefficient-extraction operator $U_m$
on Laurent series by
\begin{align*}
U_m{\left(\sum_{n=n_0}^\infty a(n)q^n\right)}=\sum_{n=\lceil n_0/m\rceil}^\infty
a(mn)q^n.
\end{align*}
The operator retains coefficients whose indices are divisible by $m$ and divides the
corresponding exponents by $m$.

We next collect several classical theta function identities that will be used
repeatedly in the sequel.

\begin{lemma}\label{Theta-Lem}
\begin{align}
\varphi(q) &=\varphi(q^4)+2q\psi(q^8),\label{Theta-iden-1}\\
\varphi(-q^2)^2 &=\varphi(q)\varphi(-q),\label{Theta-iden-2}\\
\psi(q)^2 &=\varphi(q)\psi(q^2),\label{Theta-iden-3}\\
\varphi(q)^2 &=\varphi(q^2)^2+4q\psi(q^4)^2.\label{Theta-iden-4}
\end{align}
\end{lemma}

\begin{proof}
The identities \eqref{Theta-iden-1}--\eqref{Theta-iden-4} may be found, or derived
directly, from \cite[p.~40, Entry 25]{Ber1991}. More precisely, \eqref{Theta-iden-1}
follows from (i) and (ii) therein; \eqref{Theta-iden-2} and \eqref{Theta-iden-3}
correspond to (iii) and (iv), respectively; and \eqref{Theta-iden-4}  follows from
(v) and (vi).
\end{proof}

The preceding identities will be used to extract even-indexed coefficients and simplify
the resulting expressions. The next lemma rewrites the required theta quotients as
polynomials in the auxiliary function.

\begin{lemma}
\begin{align}
\dfrac{\varphi(-q^2)^4}{\varphi(-q)^4} &=1+8\xi.\label{Tang-iden-1}\\
\dfrac{\varphi(q)^2}{\varphi(-q)^2} &=1+8\xi,\label{Tang-iden-2}\\
\dfrac{\varphi(q^2)^2}{\varphi(-q)^2} &=1+4\xi.\label{Tang-iden-4}
\end{align}
\end{lemma}

\begin{proof}
The identities \eqref{Tang-iden-1} and \eqref{Tang-iden-2} correspond to (2.15) and
(2.25), respectively, in \cite{Tang2026a}. To establish \eqref{Tang-iden-4}, we
replace $q$ by $-q$ in \eqref{Theta-iden-4}, obtaining
\begin{align*}
\varphi(q^2)^2=\varphi(-q)^2+4q\psi(q^4)^2.
\end{align*}
Dividing both sides by $\varphi(-q)^2$ gives
\begin{align*}
\dfrac{\varphi(q^2)^2}{\varphi(-q)^2}=1+4q\dfrac{\psi(q^4)^2}{\varphi(-q)^2}=1+4\xi,
\end{align*}
which proves \eqref{Tang-iden-4}.
\end{proof}

We begin with the initial extraction for color counts congruent to one modulo four. The
extracted series is the common theta factor multiplied by a polynomial in the auxiliary
function.

\begin{lemma}\label{Fund-Lem-1}
For any $i\geq0$,
\begin{align}
U_4\big(\gamma\!\:\xi^{i}\big)=\varrho\sum_{j=\lfloor(i+3)/4\rfloor}^{4i}
\varPhi_{i,j}\,\xi^j,\label{gamma-xi-exp}
\end{align}
where the coefficients $\varPhi_{i,j}$ are determined by the bivariate generating
function
\begin{align}
\sum_{i=0}^\infty\sum_{j=0}^\infty\varPhi_{i,j}\!\:x^{i}y^{j}=\dfrac{N_\varPhi}{D},
\label{bivariable-gf-1}
\end{align}
where
\begin{align}
N_\varPhi &=1-\big(60y+704y^2+2048y^3\big)x-\big(22y+272y^2+768y^3\big)x^2
-\big(2y+24y^2+64y^3\big)x^3,\label{Def-N_varPhi}\\
D &=1-\big(128y+2560y^2+16384y^3+32768y^4\big)x
-\big(80y+1600y^2+10240y^3+20480y^4\big)x^2\nonumber\\
 &\quad-\big(16y+320y^2+2048y^3+4096y^4\big)x^3
-\big(y+20y^2+128y^3+256y^4\big)x^4.\label{Def-D}
\end{align}
\end{lemma}

\begin{proof}
We first derive a recurrence from a quartic modular equation and verify the initial
cases. Together with the initial values, the recurrence determines the coefficient
generating function and yields the bounds on the exponents. Let $X=\xi(q^4)$. Sellers
and Smoot \cite[Theorem~2.2]{SS2024} established the following modular equation:
\begin{align}
\xi^4 &-\big(128X+2560X^2+16384X^3+32768X^4\big)\xi^3
-\big(80X+1600X^2+10240X^3+20480X^4\big)\xi^2\nonumber\\
 &-\big(16X+320X^2+2048X^3+4096X^4\big)\xi
-\big(X+20X^2+128X^3+256X^4\big)=0.\label{mod-eq-SS}
\end{align}
Multiplying both sides of \eqref{mod-eq-SS} by $\gamma\!\:\xi^{i-4}$ for $i\geq4$ and
then applying the $U_4$-operator, we obtain the following recurrence
\begin{align}
U_4\big(\gamma\!\:\xi^{i}\big) &=\big(128\xi+2560\xi^2+16384\xi^3+32768\xi^4\big)
U_4\big(\gamma\!\:\xi^{i-1}\big)\nonumber\\
 &\hspace{6em}+\big(80\xi+1600\xi^2+10240\xi^3+20480\xi^4\big)
U_4\big(\gamma\!\:\xi^{i-2}\big)\nonumber\\
 &\hspace{11em}+\big(16\xi+320\xi^2+2048\xi^3+4096\xi^4\big)U_4\big(\gamma\!\:\xi^{i-3}\big)
\nonumber\\
 &\hspace{16em}+\big(\xi+20\xi^2+128\xi^3+256\xi^4\big)U_4\big(\gamma\!\:\xi^{i-4}\big).
\label{Recur-gamma-xi}
\end{align}

We next establish the following initial identities:
\begin{align}
U_4(\gamma) &=\varrho,\label{ini-case-gx-1}\\
U_4(\gamma\!\:\xi) &=\varrho\big(68\xi+1856\xi^2+14336\xi^3+32768\xi^4\big),
\label{ini-case-gx-2}\\
U_4\big(\gamma\!\:\xi^2\big) &=\varrho\big(58\xi+10032\xi^2+421120\xi^3+7720960\xi^4
+73531392\xi^5+379584512\xi^6\nonumber\\
 &\quad+1006632960\xi^7+1073741824\xi^8\big),\label{ini-case-gx-3}\\
U_4\big(\gamma\!\:\xi^3\big) &=\varrho\big(14\xi+13160\xi^2+1691840\xi^3
+85352448\xi^4+2278572032\xi^5+36643274752\xi^6\nonumber\\
 &\quad+377755795456\xi^7+2558995202048\xi^8+11343008628736\xi^9
+31679678775296\xi^{10}\nonumber\\
 &\quad+50577534877696\xi^{11}+35184372088832\xi^{12}\big).\label{ini-case-gx-4}
\end{align}
We give the calculation for \eqref{ini-case-gx-2}. The remaining initial identities are
obtained by the same extraction procedure.

Replacing $q$ by $q^2$ in \eqref{Theta-iden-3} yields
\begin{align}
\psi(q^2)^2=\varphi(q^2)\psi(q^4).\label{Theta-iden-5}
\end{align}
Using \eqref{Theta-iden-1} and \eqref{Theta-iden-2}, we may write
\begin{align}
\gamma\!\:\xi=q\dfrac{\varphi(-q^2)\psi(q^4)^2}{\varphi(-q)^3}
=q\dfrac{\varphi(q)^3\psi(q^4)^2}{\varphi(-q^2)^5}
=q\dfrac{\psi(q^4)^2}{\varphi(-q^2)^5}\big(\varphi(q^4)+2q\psi(q^8)\big)^3.
\label{PF-Lem-1-step-1}
\end{align}
Extracting the even powers and reindexing the resulting series on the right-hand side
of \eqref{PF-Lem-1-step-1}, we deduce that
\begin{align*}
U_2(\gamma\!\:\xi) &=6q\dfrac{\varphi(q^2)^2\psi(q^2)^2\psi(q^4)}{\varphi(-q)^5}
+8q^2\dfrac{\psi(q^2)^2\psi(q^4)^3}{\varphi(-q)^5}\\
 &=6q\dfrac{\varphi(q^2)^3\psi(q^4)^2}{\varphi(-q)^5}
+8q^2\dfrac{\psi(q^2)^2\psi(q^4)^3}{\varphi(-q)^5}=:6S_1+8S_2,
\end{align*}
where the second equality follows from \eqref{Theta-iden-5}. We now rewrite both terms
so that a second extraction of even-indexed coefficients can be carried out in the same
way. Applying \eqref{Theta-iden-1} and \eqref{Theta-iden-2}, we have
\begin{align}
S_1 &=q\dfrac{\varphi(q^2)^3\psi(q^4)^2}{\varphi(-q^2)^{10}}
\big(\varphi(q^4)+2q\psi(q^8)\big)^5,\label{PF-Lem-1-step-3}\\
S_2 &=q^2\dfrac{\psi(q^2)^2\psi(q^4)^3}{\varphi(-q^2)^{10}}
\big(\varphi(q^4)+2q\psi(q^8)\big)^5.\label{PF-Lem-1-step-4}
\end{align}
Extracting the even-indexed coefficients from \eqref{PF-Lem-1-step-3} and
\eqref{PF-Lem-1-step-4}, and using \eqref{Theta-iden-3} and \eqref{Theta-iden-5}, we
deduce that
\begin{align*}
U_2(S_1) &=10q\dfrac{\varphi(q)^3\varphi(q^2)^5\psi(q^4)^2}{\varphi(-q)^{10}}
+80q^2\dfrac{\varphi(q)^3\varphi(q^2)^3\psi(q^4)^4}{\varphi(-q)^{10}}
+32q^3\dfrac{\varphi(q)^3\varphi(q^2)\psi(q^4)^6}{\varphi(-q)^{10}},\\
U_2(S_2) &=q\dfrac{\varphi(q)\varphi(q^2)^7\psi(q^4)^2}{\varphi(-q)^{10}}
+40q^2\dfrac{\varphi(q)\varphi(q^2)^5\psi(q^4)^4}{\varphi(-q)^{10}}
+80q^3\dfrac{\varphi(q)\varphi(q^2)^3\psi(q^4)^6}{\varphi(-q)^{10}}.
\end{align*}
Recalling the definitions of $\varrho$ and $\xi$, and applying \eqref{Tang-iden-2}
together with \eqref{Tang-iden-4}, we conclude that
\begin{align*}
U_4(\gamma\!\:\xi) &=U_2\big(U_2(\gamma\!\:\xi)\big)\\
 &=60\varrho\!\:\xi(1+8\xi)(1+4\xi)^2+480\varrho\!\:\xi^2(1+8\xi)(1+4\xi)
+192\varrho\!\:\xi^3(1+8\xi)+8\varrho\!\:\xi(1+4\xi)^3\\
 &\quad+320\varrho\!\:\xi^2(1+4\xi)^2+640\varrho\!\:\xi^3(1+4\xi)\\
 &=\varrho\big(68\xi+1856\xi^2+14336\xi^3+32768\xi^4\big),
\end{align*}
which establishes \eqref{ini-case-gx-2}.

We now combine the recurrence with these initial values to determine the bivariate
generating function of the coefficient array. Define
$u_i=U_4\big(\gamma\!\:\xi^{i}\big)$, and set
\begin{align*}
P_1 &=128\xi+2560\xi^2+16384\xi^3+32768\xi^4,
 &P_2 &=80\xi+1600\xi^2+10240\xi^3+20480\xi^4,\\
P_3 &=16\xi+320\xi^2+2048\xi^3+4096\xi^4,
 &P_4 &=\xi+20\xi^2+128\xi^3+256\xi^4.
\end{align*}
Let $G=\sum_{i=0}^\infty u_i\!\:x^i$. The recurrence \eqref{Recur-gamma-xi}
implies that
\begin{align}
\big(1-P_1x &-P_2x^2-P_3x^3-P_4x^4\big)G\nonumber\\
 &=u_0+(u_1-u_0P_1)x+(u_2-u_1P_1-u_0P_2)x^2\nonumber\\
 &\quad+(u_3-u_2P_1-u_1P_2-u_0P_3)x^3
+(u_4-u_3P_1-u_2P_2-u_1P_3-u_0P_4)x^4.\label{G-expre-1}
\end{align}
The coefficient of $x^4$ vanishes by the recurrence. On the other hand, it follows from
\eqref{Recur-gamma-xi} together with the initial identities
\eqref{ini-case-gx-1}--\eqref{ini-case-gx-4} that there exist nonnegative integers
$\varPhi_{i,j}$ such that
\begin{align*}
u_i=\varrho\sum_{j=0}^{\infty}\varPhi_{i,j}\,\xi^j.
\end{align*}
Hence,
\begin{align*}
G=\varrho\sum_{i=0}^\infty\sum_{j=0}^{\infty}\varPhi_{i,j}\,x^i\xi^j.
\end{align*}
Dividing by the common theta factor and substituting the initial values into
\eqref{G-expre-1} gives the numerator and denominator in \eqref{Def-N_varPhi} and
\eqref{Def-D}. Replacing the auxiliary function by an independent indeterminate then
yields \eqref{bivariable-gf-1}.

It remains to bound the exponents that can occur. We proceed by induction, using the
fact that each term in the recurrence is a product of a previously constructed
polynomial and a polynomial whose exponents lie between one and four. The initial
identities \eqref{ini-case-gx-1}--\eqref{ini-case-gx-4} show that the assertion holds
for $0\leq i\leq3$. Assume now that \eqref{gamma-xi-exp} is valid for all indices
$0,1,\ldots,i-1$ with some $i\geq4$. Applying the inductive hypothesis together with
the recurrence \eqref{Recur-gamma-xi}, we obtain that
\begin{align*}
U_4\big(\gamma\!\:\xi^{i}\big) &=\varrho\big(128\xi+2560\xi^2+16384\xi^3+32768\xi^4
\big)\sum_{j=\lfloor(i+2)/4\rfloor}^{4i-4}\varPhi_{i-1,j}\,\xi^j\\
 &\quad+\varrho\big(80\xi+1600\xi^2+10240\xi^3+20480\xi^4\big)
\sum_{j=\lfloor(i+1)/4\rfloor}^{4i-8}\varPhi_{i-2,j}\,\xi^j\\
 &\quad+\varrho\big(16\xi+320\xi^2+2048\xi^3+4096\xi^4\big)
\sum_{j=\lfloor i/4\rfloor}^{4i-12}\varPhi_{i-3,j}\,\xi^j\\
 &\quad+\varrho\big(\xi+20\xi^2+128\xi^3+256\xi^4\big)
\sum_{j=\lfloor(i-1)/4\rfloor}^{4i-16}\varPhi_{i-4,j}\,\xi^j\\
 &=:\varrho\sum_{j=\lfloor(i+3)/4\rfloor}^{4i}\varPhi_{i,j}\,\xi^j,
\end{align*}
where the coefficients $\varPhi_{i,j}$ are nonnegative integers and vanish outside
the stated range. This completes the inductive step and therefore establishes
\eqref{gamma-xi-exp}.
\end{proof}

The extraction polynomials for the second color class satisfy the same recurrence but
have different initial values. Their bivariate generating function therefore has the
same denominator and a different numerator.

\begin{lemma}
For any $i\geq0$,
\begin{align}
U_4\big(\delta\!\:\xi^{i}\big)=\varrho\sum_{j=\lfloor(i+3)/4\rfloor}^{4i+2}
\varPsi_{i,j}\,\xi^j,\label{delta-xi-exp}
\end{align}
where the coefficients $\varPsi_{i,j}$ are determined by
\begin{align}
\sum_{i=0}^\infty\sum_{j=0}^\infty\varPsi_{i,j}\!\:x^{i}y^{j}=\dfrac{N_\varPsi}{D},
\end{align}
where
\begin{align}
N_\varPsi &=\big(1+96y+512y^2\big)+\big(92y+960y^2+2048y^3\big)x\nonumber\\
 &\quad+\big(26y+304y^2+768y^3\big)x^2+\big(2y+24y^2+64y^3\big)x^3,\label{Def-N_varPsi}
\end{align}
and $D$ is given in \eqref{Def-D}.
\end{lemma}

\begin{proof}
We begin by deriving a recurrence satisfied by the functions
$U_4\big(\delta\!\:\xi^{i}\big)$.
Multiplying \eqref{mod-eq-SS} by $\delta\!\:\xi^{i-4}$ for $i\geq4$ and then applying
the $U_4$-operator, we obtain that
\begin{align}
U_4\big(\delta\!\:\xi^{i}\big) &=\big(128\xi+2560\xi^2+16384\xi^3+32768\xi^4\big)
U_4\big(\delta\!\:\xi^{i-1}\big)\nonumber\\
 &\hspace{6em}+\big(80\xi+1600\xi^2+10240\xi^3+20480\xi^4\big)
U_4\big(\delta\!\:\xi^{i-2}\big)\nonumber\\
 &\hspace{11em}+\big(16\xi+320\xi^2+2048\xi^3+4096\xi^4\big)
U_4\big(\delta\!\:\xi^{i-3}\big)\nonumber\\
 &\hspace{17em}+\big(\xi+20\xi^2+128\xi^3+256\xi^4\big)
U_4\big(\delta\!\:\xi^{i-4}\big).\label{Recur-delta-xi}
\end{align}

To initiate this recurrence, we establish the following four initial cases:
\begin{align}
U_4(\delta) &=\varrho\big(1+96\xi+512\xi^2\big),\label{ini-case-dx-1}\\
U_4(\delta\!\:\xi) &=\varrho\big(220\xi+15808\xi^2+329728\xi^3
+2916352\xi^4+11534336\xi^5+16777216\xi^6\big),\label{ini-case-dx-2}\\
U_4\big(\delta\!\:\xi^2\big) &=\varrho\big(106\xi+37744\xi^2+2792192\xi^3
+88100864\xi^4+1490812928\xi^5+14873001984\xi^6\nonumber\\
 &\quad+90261422080\xi^7+327491256320\xi^8+652835028992\xi^9
+549755813888\xi^{10}\big),\label{ini-case-dx-3}\\
U_4\big(\delta\!\:\xi^3\big) &=\varrho\big(18\xi+33048\xi^2+6760256\xi^3
+510050304\xi^4+19975487488\xi^5+472637505536\xi^6\nonumber\\
 &\quad+7311577841664\xi^7+77145397264384\xi^8+565926365757440\xi^9\nonumber\\
 &\quad+2888485765644288\xi^{10}+10064929440661504\xi^{11}
+22834657485651968\xi^{12}\nonumber\\
 &\quad+30399297484750848\xi^{13}+18014398509481984\xi^{14}\big).
\label{ini-case-dx-4}
\end{align}

We verify \eqref{ini-case-dx-2} in detail; the other initial identities follow by the
same extraction procedure. By \eqref{Theta-iden-1} and \eqref{Theta-iden-2}, we have
\begin{align}
\delta\!\:\xi=q\dfrac{\varphi(-q^2)^3\psi(q^4)^2}{\varphi(-q)^5}
=q\dfrac{\varphi(q)^5\psi(q^4)^2}{\varphi(-q^2)^7}
=q\dfrac{\psi(q^4)^2}{\varphi(-q^2)^7}\big(\varphi(q^4)+2q\psi(q^8)\big)^5.
\label{PF-Lem-2-step-1}
\end{align}
Extracting the even powers and reindexing the resulting series on the right-hand side
of \eqref{PF-Lem-2-step-1} and applying \eqref{Theta-iden-5}, we deduce that
\begin{align}
U_2(\delta\!\:\xi) &=10q\dfrac{\varphi(q^2)^5\psi(q^4)^2}{\varphi(-q)^7}+80q^2
\dfrac{\varphi(q^2)^3\psi(q^4)^4}{\varphi(-q)^7}+32q^3\dfrac{\varphi(q^2)
\psi(q^4)^6}{\varphi(-q)^7}=:10S_3+80S_4+32S_5.\label{PF-Lem-2-step-2}
\end{align}
To evaluate these terms further, we again invoke \eqref{Theta-iden-1} and
\eqref{Theta-iden-2}, obtaining
\begin{align*}
S_3 &=q\dfrac{\varphi(q^2)^5\psi(q^4)^2}{\varphi(-q^2)^{14}}
\big(\varphi(q^4)+2q\psi(q^8)\big)^7,\\
S_4 &=q^2\dfrac{\varphi(q^2)^3\psi(q^4)^4}{\varphi(-q^2)^{14}}
\big(\varphi(q^4)+2q\psi(q^8)\big)^7,\\
S_5 &=q^3\dfrac{\varphi(q^2)\psi(q^4)^6}{\varphi(-q^2)^{14}}
\big(\varphi(q^4)+2q\psi(q^8)\big)^7.
\end{align*}
Expanding the binomial expressions, extracting the even-indexed coefficients, and
simplifying with \eqref{Tang-iden-2}, \eqref{Tang-iden-4} and \eqref{Theta-iden-5}, we
deduce that
\begin{align}
U_2(S_3) &=\varrho\big(14\xi+672\xi^2+11648\xi^3+94208\xi^4+360448\xi^5
+524288\xi^6\big),\label{PF-Lem-2-step-6}\\
U_2(S_4) &=\varrho\big(\xi+108\xi^2+2464\xi^3+22528\xi^4+90112\xi^5+131072\xi^6\big),
\label{PF-Lem-2-step-7}\\
U_2(S_5) &=\varrho\big(14\xi^2+504\xi^3+5376\xi^4+22528\xi^5+32768\xi^6\big).
\label{PF-Lem-2-step-8}
\end{align}
Finally, applying the $U_2$-operator to \eqref{PF-Lem-2-step-2} and substituting
\eqref{PF-Lem-2-step-6}--\eqref{PF-Lem-2-step-8}, we obtain that
\begin{align*}
U_4(\delta\!\:\xi) &=10\!\:U_2(S_3)+80\!\:U_2(S_4)+32\!\:U_2(S_5)\\
 &=\varrho\big(220\xi+15808\xi^2+329728\xi^3+2916352\xi^4+11534336\xi^5
+16777216\xi^6\big).
\end{align*}
This proves \eqref{ini-case-dx-2}.

With the initial values established, the recurrence \eqref{Recur-delta-xi} determines
the bivariate generating function as in Lemma~\ref{Fund-Lem-1}. The bounds on the
exponents follow by the same induction, completing the proof.
\end{proof}

Both initial extractions have the same polynomial form. To continue the argument, we
need to show that taking even-indexed coefficients preserves this form. The next lemma
provides the required iteration formula.

\begin{lemma}\label{Fund-Lem-3}
For any $i\geq1$,
\begin{align}
U_2\big(\varrho\!\:\xi^{i}\big)=\varrho\sum_{j=\lfloor(i+1)/2\rfloor}^{2i+1}
\varOmega_{i,j}\,\xi^j,\label{bivariable-gf-3}
\end{align}
where the coefficients $\varOmega_{i,j}$ are given by
\begin{align}
\sum_{i=1}^\infty\sum_{j=1}^\infty\varOmega_{i,j}\!\:x^{i}y^{j}=\dfrac{N_\varOmega}
{\widehat{D}},
\end{align}
where
\begin{align}
N_{\varOmega} &=\big(10y+160y^2+512y^3\big)x+\big(y+20y^2+64y^3\big)x^2,
\label{Def-N_varomega}\\
\widehat{D} &=1-\big(8y+32y^2\big)x-\big(y+4y^2\big)x^2.\label{Def-WD}
\end{align}
\end{lemma}

\begin{proof}
As before, a modular equation supplies the recurrence. In this case, the equation is
quadratic, so only two initial values are needed.

Let $\widehat{X}=\xi(q^2)$. Das, Du and Tang \cite[Lemma~2.3]{DDT2024} deduced the
following (equivalent form of) modular equation of second order:
\begin{align}
\xi^2-\big(8\widehat{X}+32\widehat{X}^2\big)\xi-\big(\widehat{X}+4\widehat{X}^2
\big)=0.\label{mod-eq-DDT}
\end{align}
Multiplying both sides of \eqref{mod-eq-DDT} by $\varrho\!\:\xi^{i-2}$ for
$i\geq3$ and then applying the $U_2$-operator, we obtain the recurrence
\begin{align}
U_2\big(\varrho\!\:\xi^{i}\big)=\big(8\xi+32\xi^2\big)U_2\big(\varrho\!\:\xi^{i-1}
\big)+\big(\xi+4\xi^2\big)U_2\big(\varrho\!\:\xi^{i-2}\big).\label{Recur-varrho-xi}
\end{align}

We next establish the following two initial cases:
\begin{align}
U_2(\varrho\!\:\xi) &=\varrho\big(10\xi+160\xi^2+512\xi^3\big),
\label{ini-case-vx-1}\\
U_2\big(\varrho\!\:\xi^2\big) &=\varrho\big(\xi+100\xi^2+1664\xi^3+9216\xi^4+16384
\xi^5\big).\label{ini-case-vx-2}
\end{align}
We verify \eqref{ini-case-vx-2} in detail; \eqref{ini-case-vx-1} follows by the same
extraction procedure. By \eqref{Theta-iden-1} and \eqref{Theta-iden-2}, we have
\begin{align}
\varrho\!\:\xi^2=q^2\dfrac{\varphi(q)\varphi(q^2)\psi(q^4)^4}{\varphi(-q)^6}
=q^2\dfrac{\varphi(q^2)\psi(q^4)^4}{\varphi(-q^2)^{12}}
\big(\varphi(q^4)+2q\psi(q^8)\big)^7.\label{PF-Lem-3-step-1}
\end{align}
Extracting the even powers and reindexing the resulting series on the right-hand side
of \eqref{PF-Lem-3-step-1} and applying \eqref{Theta-iden-5}, we deduce, after
simplification, that
\begin{align*}
U_2\big(\varrho\!\:\xi^2\big) &=q\dfrac{\varphi(q)\varphi(q^2)^9\psi(q^4)^2}
{\varphi(-q)^{12}}+84q^2\dfrac{\varphi(q)\varphi(q^2)^7\psi(q^4)^4}{\varphi(-q)^{12}}\\
 &\quad+560q^3\dfrac{\varphi(q)\varphi(q^2)^5\psi(q^4)^6}{\varphi(-q)^{12}}
+448q^4\dfrac{\varphi(q)\varphi(q^2)^3\psi(q^4)^8}{\varphi(-q)^{12}}\\
 &=\varrho\!\:\xi(1+4\xi)^4+84\varrho\!\:\xi^2(1+4\xi)^3
+560\varrho\!\:\xi^3(1+4\xi)^2+448\varrho\!\:\xi^4(1+4\xi)\\
 &=\varrho\big(\xi+100\xi^2+1664\xi^3+9216\xi^4+16384\xi^5\big),
\end{align*}
where the penultimate step follows from \eqref{Tang-iden-4}. This establishes
\eqref{ini-case-vx-2}.

It remains to derive the bivariate generating function for the coefficients
$\varOmega_{i,j}$. For this purpose, define $v_i=U_2\big(\varrho\!\:\xi^{i}\big)$,
together with
\begin{align*}
\widehat{P}_1=8\xi+32\xi^2,\qquad\qquad\widehat{P}_2=\xi+4\xi^2,
\end{align*}
and let $\widehat{G}=\sum_{i=1}^\infty v_i\!\:x^i$. Using the recurrence
\eqref{Recur-varrho-xi}, we obtain that
\begin{align*}
\big(1-\widehat{P}_1x-\widehat{P}_2x^2\big)\widehat{G}=v_1x
+(v_2-v_1\widehat{P}_1)x^2.
\end{align*}
Consequently,
\begin{align*}
\widehat{G}=\dfrac{v_1x+(v_2-v_1\widehat{P}_1)x^2}
{1-\widehat{P}_1x-\widehat{P}_2x^2}.
\end{align*}
On the other hand, by the definition of the coefficients $\varOmega_{i,j}$,
\begin{align*}
\widehat{G}=\varrho\sum_{i=1}^\infty\sum_{j=1}^\infty\varOmega_{i,j}\!\:x^{i}\xi^{j}.
\end{align*}
Comparing the above two expressions for $\widehat{G}$ immediately yields the explicit
forms of $N_{\varOmega}$ and $\widehat{D}$ given in \eqref{Def-N_varomega} and
\eqref{Def-WD}, respectively.

The bounds on the exponents follow by induction from the recurrence
\eqref{Recur-varrho-xi} and the initial identities \eqref{ini-case-vx-1} and
\eqref{ini-case-vx-2}.
\end{proof}

\begin{remark}
The initial polynomial representations also contain a constant term, whose contribution
to the generating function is the common theta factor. To account for this term under a
further extraction, we record the identity
\begin{align}
U_2(\varrho)=\varrho(1+16\xi).\label{ini-case-vx-0}
\end{align}
\end{remark}

\section{Differences of generating functions}\label{Sect:DGF}

We now compare the generating functions obtained after the initial extraction and after
one further extraction. In each color class, their common theta factor can be removed,
leaving two polynomials with the same constant term. Their difference supplies the
initial data for the recurrence. For any $k\geq0$ and $i\geq1$, define
\begin{align*}
\widehat{\varPhi}_{k,i} &=\sum_{j=0}^{k}2^{3j}\binom{k}{j}\!\:\varPhi_{j,i},&
\widetilde{\varPhi}_{k,i} &=\sum_{j=1}^{4k}\widehat{\varPhi}_{k,j}\!\:\varOmega_{j,i}
+16\!\:\delta_{1,i},&\\
\widehat{\varPsi}_{k,i} &=\sum_{j=0}^{k}2^{3j}\binom{k}{j}\!\:\varPsi_{j,i},&
\widetilde{\varPsi}_{k,i} &=\sum_{j=1}^{4k+2}\widehat{\varPsi}_{k,j}\!\:\varOmega_{j,i}
+16\!\:\delta_{1,i},&
\end{align*}
where $\delta_{i,j}$ denotes the Kronecker delta function, which is defined as
$\delta_{i,j}=1$ if $i=j$ and $\delta_{i,j}=0$ otherwise. Empty sums are interpreted
as zero. The additional term in each tilded coefficient comes from extracting the
common theta factor associated with the constant term.

Repeated extraction extends this initial comparison to subsequences obtained by further
doubling the indices.

\begin{theorem}\label{Diff-THM-1}
For any $k\geq0$ and $m\geq1$,
\begin{align}
\sum_{n=0}^\infty\big(\overline{po}_{4k+1}\big(2^{m+2}n\big)-\overline{po}_{4k+1}
\big(2^{m+1}n\big)\big)q^n=\varrho\sum_{i=1}^{(4k+1)\times2^m-1}
\widehat{x}_{m,i}^{(k)}\,\xi^i,\label{Diff-gf-1}
\end{align}
where $\varrho$ and $\xi$ are defined in \eqref{Four-def-2}, and the coefficient vectors $\widehat{\textup{\textbf{x}}}_m^{(k)}=(\widehat{x}_{m,1}^{(k)},\widehat{x}_{m,2}^{(k)},
\ldots)$ are defined recursively by
\begin{align*}
\widehat{\textup{\textbf{x}}}_1^{(k)}=(\widetilde{\varPhi}_{k,i}
-\widehat{\varPhi}_{k,i})_{i\geq1},
\end{align*}
and for any $m\geq1$,
\begin{align}
\widehat{\textup{\textbf{x}}}_{m+1}^{(k)}=\widehat{\textup{\textbf{x}}}_m^{(k)}A,
\label{Def-whx-mat}
\end{align}
where $A=(\varOmega_{i,j})_{i,j\geq1}$ is the infinite matrix determined by
\eqref{bivariable-gf-3}.
\end{theorem}

\begin{proof}
We first establish the initial comparison. By \eqref{Def-NRGTF} and \eqref{Tang-iden-1},
the generating function admits the decomposition
\begin{align}
\sum_{n=0}^\infty\overline{po}_{4k+1}(n)q^n=\dfrac{\varphi(-q^2)^{4k+1}}
{\varphi(-q)^{4k+1}}=\gamma(1+8\xi)^k=\gamma\sum_{i=0}^k2^{3i}
\binom{k}{i}\!\:\xi^i.\label{Initi-gf-1}
\end{align}
Applying the $U_4$-operator to both sides of \eqref{Initi-gf-1} and using
\eqref{gamma-xi-exp}, we obtain that
\begin{align}
\sum_{n=0}^\infty\overline{po}_{4k+1}(4n)q^n &=\sum_{i=0}^k2^{3i}\binom{k}{i}\!\:
U_4\big(\gamma\!\:\xi^{i}\big)\nonumber\\
 &=\varrho\sum_{i=0}^{k}2^{3i}\binom{k}{i}\sum_{j=\lfloor(i+3)/4\rfloor}^{4i}
\varPhi_{i,j}\,\xi^j\nonumber\\
 &=\varrho\sum_{j=0}^{4k}{\Bigg(\sum_{i=0}^{k}2^{3i}\binom{k}{i}\!\:\varPhi_{i,j}
\Bigg)}\!\:\xi^j\nonumber\\
 &=\varrho+\varrho\sum_{i=1}^{4k}\widehat{\varPhi}_{k,i}\,\xi^i.\label{U_4-gf-1}
\end{align}
Next, we apply the $U_2$-operator to \eqref{U_4-gf-1}. Combining \eqref{bivariable-gf-3}
with \eqref{ini-case-vx-0}, we deduce that
\begin{align}
\sum_{n=0}^\infty\overline{po}_{4k+1}(8n)q^n &=U_2(\varrho)+\sum_{i=1}^{4k}
\widehat{\varPhi}_{k,i}\,U_2\big(\varrho\!\:\xi^i\big)\nonumber\\
 &=\varrho+16\varrho\!\:\xi+\varrho\sum_{i=1}^{4k}\widehat{\varPhi}_{k,i}
\sum_{j=\lfloor(i+1)/2\rfloor}^{2i+1}\varOmega_{i,j}\,\xi^j\nonumber\\
 &=\varrho+16\varrho\!\:\xi+\varrho\sum_{j=1}^{8k+1}\Bigg(\sum_{i=1}^{4k}
\widehat{\varPhi}_{k,i}\!\:\varOmega_{i,j}\Bigg)\!\:\xi^j\nonumber\\
 &=\varrho+\varrho\sum_{i=1}^{8k+1}\widetilde{\varPhi}_{k,i}\,\xi^i.\label{U_2-gf-1}
\end{align}
Subtracting \eqref{U_4-gf-1} from \eqref{U_2-gf-1} yields
\begin{align*}
\sum_{n=0}^\infty\overline{po}_{4k+1}(8n)q^n-\sum_{n=0}^\infty\overline{po}_{4k+1}
(4n)q^n=\varrho\sum_{i=1}^{8k+1}\widehat{x}_{1,i}^{(k)}\,\xi^i.
\end{align*}
This proves \eqref{Diff-gf-1} in the case $m=1$.

A further extraction of even-indexed coefficients gives the difference at the next
stage. Assume that \eqref{Diff-gf-1} holds for some $m\geq1$. Applying the
$U_2$-operator to both sides of \eqref{Diff-gf-1} and using \eqref{bivariable-gf-3},
we obtain that
\begin{align*}
\sum_{n=0}^\infty &\,\big(\overline{po}_{4k+1}\big(2^{m+3}n\big)-\overline{po}_{4k+1}
\big(2^{m+2}n\big)\big)q^n\\
 &\qquad\qquad\qquad\qquad=\sum_{i=1}^{(4k+1)\times2^m-1}\widehat{x}_{m,i}^{(k)}\,
U_2\big(\varrho\!\:\xi^i\big)\\
 &\qquad\qquad\qquad\qquad=\varrho\sum_{i=1}^{(4k+1)\times2^m-1}
\widehat{x}_{m,i}^{(k)}\sum_{j=\lfloor(i+1)/2\rfloor}^{2i+1}\varOmega_{i,j}\,\xi^j\\
 &\qquad\qquad\qquad\qquad=\varrho\sum_{j=1}^{(4k+1)\times2^{m+1}-1}
{\left(\sum_{i=1}^{(4k+1)\times2^m-1}\widehat{x}_{m,i}^{(k)}\!\:\varOmega_{i,j}
\right)}\!\:\xi^j\\
 &\qquad\qquad\qquad\qquad=\varrho\sum_{i=1}^{(4k+1)\times2^{m+1}-1}
\widehat{x}_{m+1,i}^{(k)}\,\xi^i.
\end{align*}
This completes the induction. The argument also includes the single-color case.
\end{proof}

The same recurrence, with a different initial vector, gives the corresponding
representation for the second color class.

\begin{theorem}
For any $k\geq0$ and $m\geq1$,
\begin{align}
\sum_{n=0}^\infty\big(\overline{po}_{4k+3}\big(2^{m+2}n\big)-\overline{po}_{4k+3}
\big(2^{m+1}n\big)\big)q^n=\varrho\sum_{i=1}^{(4k+3)\times2^m-1}
\widetilde{x}_{m,i}^{(k)}\,\xi^i,\label{Diff-gf-2}
\end{align}
where $\varrho$ and $\xi$ are defined in \eqref{Four-def-2}, and the coefficient
vectors $\widetilde{\textup{\textbf{x}}}_m^{(k)}=(\widetilde{x}_{m,1}^{(k)},
\widetilde{x}_{m,2}^{(k)},\ldots)$ are defined recursively by
\begin{align*}
\widetilde{\textup{\textbf{x}}}_1^{(k)}=(\widetilde{\varPsi}_{k,i}
-\widehat{\varPsi}_{k,i})_{i\geq1},
\end{align*}
and for any $m\geq1$,
\begin{align*}
\widetilde{\textup{\textbf{x}}}_{m+1}^{(k)}=\widetilde{\textup{\textbf{x}}}_m^{(k)}A,
\end{align*}
where $A=(\varOmega_{i,j})_{i,j\geq1}$ is the infinite matrix determined by
\eqref{bivariable-gf-3}.
\end{theorem}

\begin{proof}
Applying the $U_4$-operator to
\begin{align*}
\sum_{n=0}^\infty\overline{po}_{4k+3}(n)q^n=\delta(1+8\xi)^k=\delta\sum_{i=0}^k2^{3i}
\binom{k}{i}\!\:\xi^i,
\end{align*}
and using \eqref{delta-xi-exp}, we find that
\begin{align}
\sum_{n=0}^\infty\overline{po}_{4k+3}(4n)q^n=\varrho+\varrho\sum_{i=1}^{4k+2}
\widehat{\varPsi}_{k,i}\,\xi^i.\label{U_4-gf-2}
\end{align}
A further application of the $U_2$-operator to \eqref{U_4-gf-2} gives
\begin{align}
\sum_{n=0}^\infty\overline{po}_{4k+3}(8n)q^n=\varrho+\varrho\sum_{i=1}^{8k+5}
\widetilde{\varPsi}_{k,i}\,\xi^i.\label{U_2-gf-2}
\end{align}
Subtracting \eqref{U_4-gf-2} from \eqref{U_2-gf-2} proves the initial case.

Applying the extraction operator and using Lemma~\ref{Fund-Lem-3}, we obtain the
required identity at the next stage:
\begin{align*}
\sum_{n=0}^\infty\big(\overline{po}_{4k+3}\big(2^{m+3}n\big)-\overline{po}_{4k+3}
\big(2^{m+2}n\big)\big)q^n=\varrho\sum_{i=1}^{(4k+3)\times2^{m+1}-1}
\widetilde{x}_{m+1,i}^{(k)}\,\xi^i.
\end{align*}
This completes the inductive step and hence the proof.
\end{proof}

\section{Proof of Theorem \ref{Main-THM}}\label{Sect:PT}

The preceding expressions reduce the main theorem to divisibility estimates for the
polynomial coefficients. We first estimate the coefficient arrays and then transfer
these estimates to the recursively defined vectors. Finally, we extract the required
subsequences.

\subsection{2-Adic valuations}
For any integer $n$, let $\nu(n)$ denote the highest power of $2$ that divides $n$,
with the convention that $\nu(0)=\infty$. We begin by establishing lower bounds
for the $2$-adic valuations of the coefficients $\varPhi_{i,j}$, $\varPsi_{i,j}$ and
$\varOmega_{i,j}$.

\begin{lemma}\phantomsection\label{2-Adic-Lem}
\begin{enumerate}[label=\textup{(\roman*)}, leftmargin=*]
\item For any $i\geq0$ and $j\geq0$,
\begin{align}
\nu(\varPhi_{i,j}) &\geq4j-2i,\label{2-adid-varPhi}\\
\nu(\varPsi_{i,j}) &\geq4j-2i.\label{2-adid-varPsi}
\end{align}
\item For any $i\geq1$ and $j\geq1$,
\begin{align}
\nu(\varOmega_{i,j}) &\geq2j-i.\label{2-adid-varOmega}
\end{align}
\end{enumerate}
\end{lemma}

\begin{proof}
Throughout the following recurrences, entries outside the defining index ranges of the
coefficient arrays are understood to be zero.

(i) We first consider the coefficients $\varPhi_{i,j}$. Combining \eqref{Recur-gamma-xi}
with the initial identities \eqref{ini-case-gx-1}--\eqref{ini-case-gx-4}, we record the
initial rows and the coefficient recurrence needed for the induction.
\begin{enumerate}[start=1, leftmargin=*]
\item The first four rows of $(\varPhi_{i,j})_{i,j\geq0}$ are given by
\begin{align}
\big( &1,\,0,\,0,\,0,\ldots\big),\label{varPhi-r1}\\[2pt]
\big( &0,\,17\times2^2,\,29\times2^6,\,7\times2^{11},\,2^{15},\,0,\,0,\,0,\ldots
\big),\label{varPhi-r2}\\[2pt]
\big( &0,\,29\times2,\,627\times2^4,\,1645\times2^8,\,1885\times2^{12},\,
561\times2^{17},\,181\times2^{21},\,15\times2^{26},\,2^{30},\,0,\,0,\,0,\ldots
\big),\label{varPhi-r3}\\[2pt]
\big ( &0,\,7\times2,\,1645\times2^3,\,26435\times2^6,\,10419\times2^{13},\,
139073\times2^{14},\,139783\times2^{18},\,5629\times2^{26},\nonumber\\
 &9533\times2^{28},\,2641\times2^{32},\,
461\times2^{36},\,23\times2^{41},\,2^{45},\,0,\,0,\,0,\ldots\big).\label{varPhi-r4}
\end{align}
\item For any $i\geq4$, $\varPhi_{i,0}=0$; and for any $j\geq1$,
\begin{align}
\varPhi_{i,j} &=128\!\:\varPhi_{i-1,j-1}+80\!\:\varPhi_{i-2,j-1}
+16\!\:\varPhi_{i-3,j-1}+\varPhi_{i-4,j-1}+2560\!\:\varPhi_{i-1,j-2}\nonumber\\
 &\quad+1600\!\:\varPhi_{i-2,j-2}+320\!\:\varPhi_{i-3,j-2}+20\!\:\varPhi_{i-4,j-2}
+16384\!\:\varPhi_{i-1,j-3}\nonumber\\
 &\quad+10240\!\:\varPhi_{i-2,j-3}+2048\!\:\varPhi_{i-3,j-3}
+128\!\:\varPhi_{i-4,j-3}+32768\!\:\varPhi_{i-1,j-4}\nonumber\\
 &\quad+20480\!\:\varPhi_{i-2,j-4}+4096\!\:\varPhi_{i-3,j-4}
+256\!\:\varPhi_{i-4,j-4}.\label{Recur-varPhi}
\end{align}
\end{enumerate}
It follows immediately from \eqref{varPhi-r1}--\eqref{varPhi-r4} that
\begin{align*}
(\nu(\varPhi_{0,j}))_{j\geq0}=( &0,\infty,\infty,\infty,\ldots),\\
(\nu(\varPhi_{1,j}))_{j\geq0}=( &\infty,2,6,11,15,\infty,\infty,\infty,\ldots),\\
(\nu(\varPhi_{2,j}))_{j\geq0}=( &\infty,1,4,8,12,17,21,26,30,\infty,\infty,\infty,
\ldots),\\
(\nu(\varPhi_{3,j}))_{j\geq0}=( &\infty,1,3,6,13,14,18,26,28,32,36,41,45,\infty,
\infty,\infty,\ldots).
\end{align*}
Hence, \eqref{2-adid-varPhi} is valid for $0\leq i\leq3$.

We proceed by induction on the row index. Fix $i\geq4$, and suppose that the bound
holds for all preceding rows. By the inductive hypothesis, the terms appearing on the
right-hand side of \eqref{Recur-varPhi} satisfy
\begin{align*}
\nu(\varPhi_{i-1,j-1}) &\geq4j-2i-2,& \nu(\varPhi_{i-2,j-1}) &\geq4j-2i,&\\
\nu(\varPhi_{i-3,j-1}) &\geq4j-2i+2,& \nu(\varPhi_{i-4,j-1}) &\geq4j-2i+4,&\\
\nu(\varPhi_{i-1,j-2}) &\geq4j-2i-6,& \nu(\varPhi_{i-2,j-2}) &\geq4j-2i-4,&\\
\nu(\varPhi_{i-3,j-2}) &\geq4j-2i-2,& \nu(\varPhi_{i-4,j-2}) &\geq4j-2i,&\\
\nu(\varPhi_{i-1,j-3}) &\geq4j-2i-10,& \nu(\varPhi_{i-2,j-3}) &\geq4j-2i-8,&\\
\nu(\varPhi_{i-3,j-3}) &\geq4j-2i-6,& \nu(\varPhi_{i-4,j-3}) &\geq4j-2i-4,&\\
\nu(\varPhi_{i-1,j-4}) &\geq4j-2i-14,& \nu(\varPhi_{i-2,j-4}) &\geq4j-2i-12,&\\
\nu(\varPhi_{i-3,j-4}) &\geq4j-2i-10,& \nu(\varPhi_{i-4,j-4}) &\geq4j-2i-8.
\end{align*}
Using the recurrence \eqref{Recur-varPhi}, we therefore obtain that
\begin{align*}
\nu(\varPhi_{i,j}) &\geq\begin{aligned}[t]
\!\min\{&\nu(\varPhi_{i-1,j-1})+7,\,\nu(\varPhi_{i-2,j-1})+4,\,
\nu(\varPhi_{i-3,j-1})+4,\,\nu(\varPhi_{i-4,j-1})+0,\\
 &\nu(\varPhi_{i-1,j-2})+9,\,\nu(\varPhi_{i-2,j-2})+6,\,\nu(\varPhi_{i-3,j-2})+6,\,
\nu(\varPhi_{i-4,j-2})+2,\\
 &\nu(\varPhi_{i-1,j-3})+14,\,\nu(\varPhi_{i-2,j-3})+11,\,\nu(\varPhi_{i-3,j-3})+11,\,
\nu(\varPhi_{i-4,j-3})+7,\\
 &\nu(\varPhi_{i-1,j-4})+15,\,\nu(\varPhi_{i-2,j-4})+12,\,\nu(\varPhi_{i-3,j-4})+12,\,
\nu(\varPhi_{i-4,j-4})+8\}\end{aligned}\\
 &\geq4j-2i.
\end{align*}
This completes the inductive step. Hence, \eqref{2-adid-varPhi} holds for any $i\geq0$
and $j\geq0$.

We next consider the coefficients $\varPsi_{i,j}$. Combining the recurrence
\eqref{Recur-delta-xi} with the initial cases
\eqref{ini-case-dx-1}--\eqref{ini-case-dx-4}, we obtain the following description of
the coefficient array.
\begin{enumerate}[start=3, leftmargin=*]
\item The first four rows of $(\varPsi_{i,j})_{i,j\geq0}$ are given by
\begin{align}
\big( &1,\,3\times2^5,\,2^9,\,0,\,0,\,0,\ldots\big),\label{varPsi-r1}\\[2pt]
\big( &0,\,55\times2^2,\,247\times2^6,\,161\times2^{11},\,89\times2^{15},\,
11\times2^{20},\,2^{24},\,0,\,0,\,0,\ldots\big),\label{varPsi-r2}\\[2pt]
\big( &0,\,2\times53,\,2359\times2^4,\,10907\times2^8,\,21509\times2^{12},\,
 5687\times2^{18},\,1773\times2^{23},\,1345\times2^{26},\,305\times2^{30},\nonumber\\
 &19\times2^{35},\,
2^{39},\,0,\,0,\,0,\ldots\big),\label{varPsi-r3}\\[2pt]
\big( &0,\,9\times2,\,4131\times2^3,\,105629\times2^6,\,31131\times2^{14},\,
1219207\times2^{14},\,1802969\times2^{18},\,108951\times2^{26},\nonumber\\
 &287389\times2^{28},\,131765\times2^{32},42033\times2^{36},\,4577\times2^{41},\,
649\times2^{45},\,27\times2^{50},\,2^{54},\,0,\,0,\,0,\ldots\big).\label{varPsi-r4}
\end{align}
\item For any $i\geq4$, $\varPsi_{i,0}=0$; and for any $j\geq1$,
\begin{align}
\varPsi_{i,j} &=128\!\:\varPsi_{i-1,j-1}+80\!\:\varPsi_{i-2,j-1}
+16\!\:\varPsi_{i-3,j-1}+\varPsi_{i-4,j-1}+2560\!\:\varPsi_{i-1,j-2}\nonumber\\
 &\quad+1600\!\:\varPsi_{i-2,j-2}+320\!\:\varPsi_{i-3,j-2}
+20\!\:\varPsi_{i-4,j-2}+16384\!\:\varPsi_{i-1,j-3}\nonumber\\
 &\quad+10240\!\:\varPsi_{i-2,j-3}+2048\!\:\varPsi_{i-3,j-3}
+128\!\:\varPsi_{i-4,j-3}+32768\!\:\varPsi_{i-1,j-4}\nonumber\\
 &\quad+20480\!\:\varPsi_{i-2,j-4}+4096\!\:\varPsi_{i-3,j-4}
+256\!\:\varPsi_{i-4,j-4}.\label{Recur-varPsi}
\end{align}
\end{enumerate}
The initial rows satisfy the same valuation bound, and the recurrence has the same
coefficients. The preceding induction therefore proves \eqref{2-adid-varPsi}.

(ii) We finally estimate the $2$-adic valuation of the iteration coefficients. For the
iteration coefficients, we use a second order linear recurrence together with the first
two rows.

\begin{enumerate}[start=5, leftmargin=*]
\item The first two rows of $(\varOmega_{i,j})_{i,j\geq1}$ are given by
\begin{align}
\big( &5\times2,\,5\times2^5,\,2^9,\,0,\,0,\,0,\ldots\big),\label{varOmega-r1}\\[2pt]
\big( &1,\,5^2\times2^2,\,13\times2^7,\,9\times2^{10},\,2^{14},\,0,\,0,\,0,
\ldots\big).\label{varOmega-r2}
\end{align}
\item For any $i\geq3$, $\varOmega_{i,1}=0$; and for any $j\geq2$,
\begin{align}
\varOmega_{i,j} &=8\!\:\varOmega_{i-1,j-1}+\varOmega_{i-2,j-1}
+32\!\:\varOmega_{i-1,j-2}+4\!\:\varOmega_{i-2,j-2}.\label{Recur-varOmega}
\end{align}
\end{enumerate}
The first two rows \eqref{varOmega-r1} and \eqref{varOmega-r2} therefore have the
following valuation sequences, indexed from the first column:
\begin{align*}
(\nu(\varOmega_{1,j}))_{j\geq1}=( &1,5,9,\infty,\infty,\infty,\ldots),\\
(\nu(\varOmega_{2,j}))_{j\geq1}=( &0,2,7,10,14,\infty,\infty,\infty,\ldots).
\end{align*}
This establishes \eqref{2-adid-varOmega} for $i=1$ and $i=2$.

Assume that \eqref{2-adid-varOmega} is valid for all $1,2,\ldots,i-1$ with some
$i\geq3$. Using the recurrence \eqref{Recur-varOmega} together with the inductive
hypothesis, we obtain that
\begin{align*}
\nu(\varOmega_{i,j}) &\geq\min\{\nu(\varOmega_{i-1,j-1})+3,\,
\nu(\varOmega_{i-2,j-1})+0,\,\nu(\varOmega_{i-1,j-2})+5,\,
\nu(\varOmega_{i-2,j-2})+2\}\\
 &\geq2j-i.
\end{align*}
This completes the inductive step. Therefore, by induction, \eqref{2-adid-varOmega}
holds for any $i\geq1$ and $j\geq1$.
\end{proof}

The preceding bounds control both the initial data and the common recurrence. We now
show that each iteration increases the guaranteed divisibility of the coefficient
vectors.

\begin{theorem}
For any $k\geq0$, $m\geq1$ and $i\geq1$,
\begin{align}
\nu(\widehat{x}_{m,i}^{(k)}) &\geq m+2i+1,\label{2-adic-est-whx}\\
\nu(\widetilde{x}_{m,i}^{(k)}) &\geq m+2i+1.\label{2-adic-est-wtx}
\end{align}
\end{theorem}

\begin{proof}
To prove \eqref{2-adic-est-whx}, we first estimate the initial vector. The contribution
from the extraction matrix satisfies a uniform valuation bound, while the additional
constant-term contribution affects only the first coefficient. By \eqref{2-adid-varPhi}
and \eqref{2-adid-varOmega}, we deduce that for any $k\geq0$ and $i\geq1$,
\begin{align}
\nu(\widehat{\varPhi}_{k,i}) &\geq\min_{0\leq j\leq k}\{3j+\nu(\varPhi_{j,i})\}\geq4i.
\label{Init-est-1}
\end{align}
The estimate for the iteration coefficients, together with the additional contribution
to the first coefficient, gives
\begin{align}
\nu(\widetilde{\varPhi}_{k,i}) &\geq\begin{cases}
\;4 &\quad\textrm{if}\;\;i=1,\\[2pt]
2i+3 &\quad\textrm{if}\;\;i\geq2.
\end{cases}\label{Init-est-2}
\end{align}
The bounds \eqref{Init-est-1} and \eqref{Init-est-2} imply the required estimate for
the initial difference vector:
\begin{align*}
\nu(\widehat{x}_{1,i}^{(k)})\geq2i+2.
\end{align*}
Hence, \eqref{2-adic-est-whx} holds for the initial case $m=1$.

Assume now that \eqref{2-adic-est-whx} is valid for some $m\geq1$. Using the inductive
hypothesis together with the recurrence \eqref{Def-whx-mat}, we deduce that
\begin{align*}
\nu(\widehat{x}_{m+1,i}^{(k)}) &\geq\min_{1\leq j\leq(4k+1)\times2^m-1}
\big\{\nu(\widehat{x}_{m,j}^{(k)})+\nu(\varOmega_{j,i})\big\}\\
 &\geq\min_{1\leq j\leq(4k+1)\times2^m-1}\{m+2j+1+2i-j\}\\
 &=\min_{1\leq j\leq(4k+1)\times2^m-1}\{m+2i+j+1\}\\
 &=m+2i+2.
\end{align*}
This establishes the inductive step and therefore proves \eqref{2-adic-est-whx} for
any $m\geq1$.

Applying \eqref{2-adid-varPsi} in place of \eqref{2-adid-varPhi} gives the same bound
for the second initial vector. Since both families satisfy the same recurrence, the
preceding induction also proves \eqref{2-adic-est-wtx}.
\end{proof}

\subsection{Proof of \texorpdfstring{\eqref{Int-cong-fam-1}--\eqref{Int-cong-fam-6}}{}}

The valuation bounds, together with the integrality of the theta factor and the
auxiliary function, imply \eqref{Int-cong-fam-1} through the representations in
Section~\ref{Sect:DGF}. To prove the refinements, we reduce these representations
modulo the required powers of two. Only the linear term in the auxiliary variable
remains, so it suffices to analyze a single theta quotient.

A direct consequence of \eqref{2-adic-est-whx} and \eqref{2-adic-est-wtx} is that
\begin{align}
\nu(\widehat{x}_{m,i}^{(k)}) &\geq\begin{cases}
m+3 &\quad\textrm{if}\;\;i=1,\\[2pt]
m+5 &\quad\textrm{if}\;\;i\geq2.
\end{cases}\label{Fur-est-1}\\[3pt]
\nu(\widetilde{x}_{m,i}^{(k)}) &\geq\begin{cases}
m+3 &\quad\textrm{if}\;\;i=1,\\[2pt]
m+5 &\quad\textrm{if}\;\;i\geq2.
\end{cases}\label{Fur-est-2}
\end{align}
Applying the estimates \eqref{Fur-est-1} and \eqref{Fur-est-2} to the expansions
\eqref{Diff-gf-1} and \eqref{Diff-gf-2}, we obtain the reduced congruence relation
\begin{align}
\sum_{n=0}^\infty\big(\overline{po}_{2k+1}\big(2^{m+2}n\big)-\overline{po}_{2k+1}
\big(2^{m+1}n\big)\big)q^n\equiv y_{m,1}^{(k)}\,\varrho\!\:\xi\pmod{2^{m+5}},
\label{Red-cong-relat}
\end{align}
where the scalar coefficient is the coefficient of the linear term in the corresponding
color class and is divisible by $2^{m+3}$. Both color classes now reduce to the same
theta quotient. Repeated separation of even- and odd-indexed coefficients yields the refinements on the required residue classes modulo four and eight. The two remaining refinements are obtained from a $5$-dissection. Using \eqref{Theta-iden-1} and
\eqref{Theta-iden-2}, we have
\begin{align}
\varrho\!\:\xi=q\dfrac{\varphi(q)\varphi(q^2)\psi(q^4)^2}{\varphi(-q)^4}\equiv
q\dfrac{\varphi(q^2)\psi(q^4)^2}{\varphi(-q^2)^2}\big(\varphi(q^4)+2q\psi(q^8)\big)
\pmod{4}.\label{PF-MT-step-1}
\end{align}
Substituting \eqref{PF-MT-step-1} into \eqref{Red-cong-relat}, we extract the
odd-indexed coefficients and reindex the resulting series. Applying \eqref{Theta-iden-1}
and \eqref{Theta-iden-2} then gives
\begin{align*}
\sum_{n=0}^\infty &\,\big(\overline{po}_{2k+1}\big(2^{m+2}(2n+1)\big)
-\overline{po}_{2k+1}\big(2^{m+1}(2n+1)\big)\big)q^n\\
 &\equiv y_{m,1}^{(k)}\dfrac{\varphi(q)\varphi(q^2)\psi(q^2)^2}{\varphi(-q)^2}
=y_{m,1}^{(k)}\dfrac{\varphi(q^2)\psi(q^2)^2}{\varphi(-q^2)^4}
\big(\varphi(q^4)+2q\psi(q^8)\big)^3\pmod{2^{m+5}}.
\end{align*}
We next separate the even- and odd-indexed coefficients and reindex the resulting
series. Using the identities in Lemma~\ref{Theta-Lem} and simplifying the resulting
theta quotients modulo $4$, we obtain that
\begin{align}
\sum_{n=0}^\infty &\,\big(\overline{po}_{2k+1}\big(2^{m+2}(4n+1)\big)
-\overline{po}_{2k+1}\big(2^{m+1}(4n+1)\big)\big)q^n\nonumber\\
 &\equiv y_{m,1}^{(k)}\dfrac{\varphi(q)\varphi(q^2)^3\psi(q)^2}{\varphi(-q^2)^2}
=y_{m,1}^{(k)}\dfrac{\varphi(q)^2\varphi(q^2)^3\psi(q^2)}{\varphi(-q^2)^2}\nonumber\\
 &=y_{m,1}^{(k)}\dfrac{\varphi(q^2)^3\psi(q^2)}{\varphi(-q^2)^2}
\big(\varphi(q^2)^2+4q\psi(q^4)^2\big)\pmod{2^{m+5}},\label{PF-MT-step-2}\\[3pt]
\sum_{n=0}^\infty &\,\big(\overline{po}_{2k+1}\big(2^{m+2}(4n+3)\big)
-\overline{po}_{2k+1}\big(2^{m+1}(4n+3)\big)\big)q^n\nonumber\\
 &\equiv2y_{m,1}^{(k)}\dfrac{\varphi(q)\varphi(q^2)^2\psi(q)^2\psi(q^4)}
{\varphi(-q^2)^2}\equiv2y_{m,1}^{(k)}\dfrac{\varphi(q^2)^3\psi(q^2)\psi(q^4)}
{\varphi(-q^2)^2}\pmod{2^{m+5}}.\label{PF-MT-step-3}
\end{align}
The prefactor in \eqref{PF-MT-step-3} supplies the additional factor of two required
for \eqref{Int-cong-fam-2}. In \eqref{PF-MT-step-2}, the odd-indexed coefficients
acquire a further factor of four, yielding \eqref{Int-cong-fam-5}. Finally, the last
expression in \eqref{PF-MT-step-3} contains only even powers of $q$, so its
odd-indexed coefficients vanish modulo the stated power of two, proving \eqref{Int-cong-fam-6}.

For the two remaining progressions, it suffices to gain one further factor of $2$. We
therefore return to \eqref{Red-cong-relat} and reduce the accompanying theta quotient
modulo $2$. Using the product representations and the elementary congruence
$E(q)^2\equiv E(q^2)\pmod{2}$, we obtain that
\begin{align}
\sum_{n=0}^\infty\big(\overline{po}_{2k+1}\big(2^{m+2}n\big)-\overline{po}_{2k+1}
\big(2^{m+1}n\big)\big)q^n
 &\equiv y_{m,1}^{(k)}\,q\dfrac{E(q^2)^7E(q^4)E(q^8)^2}{E(q)^{10}}\nonumber\\
 &\equiv y_{m,1}^{(k)}\,qE(q^8)^3\pmod{2^{m+4}}.\label{PF-MT-step-4}
\end{align}
To identify the residue classes that occur in this product, we use the classical
$5$-dissection of Euler's product \cite[Eq.~(8.1.1)]{Hir2017},
\begin{align*}
E(q)=E(q^{25})\big(R(q^5)^{-1}-q-q^2R(q^5)\big),
\end{align*}
where
\begin{align*}
R(q)=\dfrac{(q;q^5)_\infty(q^4;q^5)_\infty}{(q^2;q^5)_\infty(q^3;q^5)_\infty}.
\end{align*}
Cubing the dissection, replacing $q$ by $q^8$, and multiplying by $q$, we deduce that
\begin{align*}
qE(q^8)^3=qE(q^{200})^3\big(R(q^{40})^{-3}-3q^8R(q^{40})^{-2}+5q^{24}-3q^{40}
R(q^{40})^2-q^{48}R(q^{40})^3\big).
\end{align*}
Including the leading factor of $q$, the expansion contains only exponents congruent
to $0$, $1$ or $4$ modulo $5$. The coefficients with indices congruent to $2$ or $3$
therefore vanish. Together with \eqref{PF-MT-step-4}, this proves \eqref{Int-cong-fam-3}
and \eqref{Int-cong-fam-4}.

\section{Concluding remarks}

The internal congruences proved here give increasingly strong divisibility bounds for
differences between values on related arithmetic progressions. A complementary question
concerns the divisibility of individual values on a set of natural density one. Using
the theory of modular forms, Lin \cite[Theorem~1.3]{Lin2012} proved that, for each
fixed $j\geq1$,
\begin{align*}
\lim_{X\rightarrow\infty}\dfrac{\#\{0\leq n<X\colon\overline{po}_2(n)\equiv0\pmod{2^j}
\}}{{X}}=1.
\end{align*}
Thus, for each fixed power of $2$, divisibility holds on a set of indices of natural
density one. Does the analogous density-one statement hold for every fixed odd number
of colors and every fixed power of $2$?

\section*{Acknowledgements}

Dazhao Tang was partially supported by the National Natural Science Foundation of China
(No.~12201093) and the Science and Technology Research Program of Chongqing Municipal
Education Commission (No.~KJQN202500501).

\bibliographystyle{amsplain}

\end{document}